\documentclass[11pt, reqno]{amsart}
\usepackage{amsmath,amsopn,amssymb,amsthm,multicol}
\usepackage[cp1251]{inputenc}
\usepackage[english]{babel}
\usepackage[breaklinks=true,colorlinks=true,linkcolor=blue,citecolor=blue,urlcolor=blue]{hyperref}
\usepackage[dvipdf]{graphicx}

\renewenvironment{proof}[1][Proof]{\textbf{#1.} }
{\ \rule{0.5em}{0.5em}}

\newtheorem{theorem}{Theorem}

\newtheorem{lemma}{Lemma}
\newtheorem{remark}{Remark}

\begin{document}

\title[On properties of the sets   \dots]
{On properties of the sets of positively curved Riemannian metrics on generalized Wallach spaces}

\author{N.\,A.~Abiev}
\address{N.\,A.~Abiev \newline
Institute of Mathematics NAS KR, Bishkek, prospect Chui, 265a, 720071, Kyrgyz Republic}
\email{abievn@mail.ru}

\begin{abstract}
Sets related to positively curved invariant Riemannian metrics on generalized Wallach spaces
are considered.
The problem arises in studying of the evolution of such metrics under the normalized Ricci flow equation.
For Riemannian metrics of the Wallach spaces
$\operatorname{SU}(3)/T_{\max}$, $\operatorname{Sp(3)}/ \left(\operatorname{Sp(1)}\right)^3$ and $F_4/\operatorname{Spin(8)}$
which admit positive sectional curvature and belong to a given
invariant  surface $\Sigma$ of the normalized Ricci flow
we established that they form a set $\Sigma S$ bounded by three connected and pairwise disjoint
regular space curves such that
each of them approaches two others asymptotically at infinity.
Analogously, we proved that for all generalized Wallach spaces
the set of Riemannian metrics $\Sigma R$
which belong to $\Sigma$ and admit positive Ricci curvature
is bounded by three curves each consisting of two connected components
as regular curves. Intersections and asymptotical behaviors of these
components were studied as well.

\vspace{2mm} \noindent Key words and phrases: Wallach space, generalized Wallach space,
Riemannian metric, normalized Ricci flow, sectional curvature, Ricci curvature, invariant set.

\vspace{2mm}

\end{abstract}

\maketitle

\section*{Introduction}\label{vvedenie}

The paper is devoted to the study of structural properties of two important sets
responsible for positivity of the sectional and the Ricci curvatures
of invariant Riemannian metrics on the Wallach spaces and generalized Wallach spaces.
The Wallach spaces
\begin{equation} \label{SWS}
W_6:=\operatorname{SU}(3)/T_{\max},
\ \ W_{12}:=\operatorname{Sp(3)}/\operatorname{Sp(1)}\times \operatorname{Sp(1)}
\times \operatorname{Sp(1)},
\ \ W_{24}:=F_4/\operatorname{Spin(8)}.
\end{equation}
are well-known and admit invariant Riemannian metrics of positive sectional curvature (see~\cite{Wal} for details).
As for generalized Wallach space, firstly, recall its definition and basic properties
(see~\cite{Nikonorov2, Nikonorov1}).
Let $G/H$ be a~homogeneous almost effective compact space  with  a (compact) semisimple connected Lie group $G$ and its closed subgroup $H$. Denote by~$\mathfrak{g}$ and~$\mathfrak{h}$ the~corresponding Lie algebras of~$G$ and~$H$.
Then $[\boldsymbol{\cdot}\,,\boldsymbol{\cdot}]$ is a corresponding~Lie bracket
of~$\mathfrak{g}$ whereas  $B(\boldsymbol{\cdot}\,,\boldsymbol{\cdot})$ is
the~Killing form of~$\mathfrak{g}$. Note that $\langle\boldsymbol{\cdot}
\,,\boldsymbol{\cdot}\rangle=-B(\boldsymbol{\cdot}\,,\boldsymbol{\cdot})$
is a~bi-invariant inner product on~$\mathfrak{g}$.
In this way  invariant Riemannian metrics on~$G/H$ can be identified with $\operatorname{Ad}(H)$-invariant inner products on the
orthogonal complement~$\mathfrak{p}$ of~$\mathfrak{h}$ in~$\mathfrak{g}$ with respect to~$\langle\boldsymbol{\cdot}\,,\boldsymbol{\cdot}\rangle$.
Compact homogeneous spaces $G/H$ whose isotropy representation admits a decomposition into a direct sum
$\mathfrak{p}=\mathfrak{p}_1\oplus \mathfrak{p}_2\oplus \mathfrak{p}_3$ of three
$\operatorname{Ad}(H)$-invariant irreducible modules
$\mathfrak{p}_1$, $\mathfrak{p}_2$ and $\mathfrak{p}_3$
satisfying
$[\mathfrak{p}_i,\mathfrak{p}_i]\subset \mathfrak{h}$ for each $i\in\{1,2,3\}$
were called {\it three-locally-symmetric spaces}\/ according to~\cite{Lomshakov2}
or {\it generalized Wallach spaces}\/  in the terminology of~\cite{Nikonorov1}.
The main  characteristic of these spaces is that every generalized Wallach space can be described by a triple of real parameters $a_i:=A/d_i\in (0,1/2]$, $i=1,2,3$,
where $d_i=\dim(\mathfrak{p}_i)$ and $A$ is some important positive constant
(see~\cite{Lomshakov2, Nikonorov2} for details).
It should be also noted that not every triple $(a_1,a_2,a_3)\in (0,1/2] \times (0,1/2] \times (0,1/2]$ corresponds to some generalized Wallach spaces.
An interesting fact is the fact that the Wallach spaces~\eqref{SWS}
are partial cases  $a_1=a_2=a_3=a$ of generalized Wallach spaces
with $a=1/6$, $a=1/8$ and $a=1/9$ respectively (see~\cite{AN}).

As noted above for a fixed bi-invariant inner product
$\langle\cdot, \cdot\rangle$ on the Lie algebra $\mathfrak{g}$ of the Lie group $G$,
any $G$-invariant Riemannian metric $\bold{g}$ on $G/H$ can  be determined by an $\operatorname{Ad} (H)$-invariant inner product
\begin{equation}\label{metric}
(\cdot, \cdot)=\left.x_1\langle\cdot, \cdot\rangle\right|_{\mathfrak{p}_1}+
\left.x_2\langle\cdot, \cdot\rangle\right|_{\mathfrak{p}_2}+
\left.x_3\langle\cdot, \cdot\rangle\right|_{\mathfrak{p}_3},
\end{equation}
where $x_1,x_2,x_3$ are positive real numbers
(a detailed exposition can be found in~\cite{Lomshakov2, Nikonorov2, Nikonorov4, Nikonorov1} and references therein).
In~\cite{Nikonorov2} the explicit expressions
$\operatorname{Ric}_{\bold{g}}=\left.{\bf r_1}\, \operatorname{Id} \right|_{\mathfrak{p}_1}+
\left.{\bf r_2}\, \operatorname{Id} \right|_{\mathfrak{p}_2}+
\left.{\bf r_3}\, \operatorname{Id} \right|_{\mathfrak{p}_3}$
and
$S_{\bold{g}}=d_1{\bf r_1}+d_2{\bf r_2}+d_3{\bf r_3}$
were derived
for the Ricci tensor $\operatorname{Ric}_{\bold{g}}$ and
the scalar curvature  $S_{\bold{g}}$  of the metric \eqref{metric}
on generalized Wallach spaces,
where
\begin{equation}\label{Ricci_princ}
{\bf r_i}:=\frac{1}{2x_i}+\frac{1}{2a_i}\left(\frac{x_i}{x_jx_k}-\frac{x_k}{x_ix_j}-
\frac{x_j}{x_ix_k} \right)
\end{equation}
are the principal Ricci curvatures,
$\{i,j,k\}=\{1,2,3\}$.

Knowing  $\operatorname{Ric}_{\bold{g}}$ and
$S_{\bold{g}}$  allowed us to initiate  in~\cite{AANS1}
the study of the normalized Ricci flow equation
\begin{equation}\label{ricciflow}
\dfrac {\partial}{\partial t} \bold{g}(t) = -2 \operatorname{Ric}_{\bold{g}}+ 2{\bold{g}(t)}\frac{S_{\bold{g}}}{n}
\end{equation}
introduced by R.~Hamilton in \cite{Ham}
on generalized Wallach spaces.
Since then studies related to this topic
were  continued in~\cite{Ab_kar3,  Ab1,   AANS2, Stat}
concerning
classifications of singular (equilibria) points of~\eqref{ricciflow}
 being Einstein metrics
and their bifurcations.
The authors of~\cite{Ab2,  Bat, Bat2} studied an interesting and quite complicated
surface of bifurcations
of~\eqref{ricciflow}
defined by a symmetric polynomial equation in three variables $a_1,a_2,a_3$ of degree~$12$.
In the sequel authors of~\cite{Ab7, AN}  considered
the evolution of positively curved Riemannian metrics under the influence
of~\eqref{three_equat} on an interesting
class of generalized  Wallach spaces with coincided parameters $a_1=a_2=a_3:=a\in (0,1/2)$
generalizing some results of~\cite{Bo, ChWal}.
In this case \eqref{ricciflow} can be  reduced to the following system
of three autonomous ordinary differential equations (see~\cite{AN}):
\begin{equation}\label{three_equat}
\frac{d}{dt}x_i(t)=f_i(x_1,x_2,x_3):=\frac{x_i}{x_j}+\frac{x_i}{x_k} + 2a\,\left(\frac{x_j}{x_k}+\frac{x_k}{x_j}-2\,\frac{x_i^2}{x_jx_k}\right)-2
\end{equation}
with $\{i,j,k\}=\{1,2,3\}$.

A relevant information on homogeneous spaces admitting positive sectional or (and) positive Ricci curvatures can be found in
\cite{AW, BB, Be, Wal, Wil, WiZi, XuWolf} and references therein.
In~\cite{AN} it was proved  that~\eqref{ricciflow} deforms
all generic metrics with positive sectional curvature into metrics with mixed sectional curvature on each Wallach space in~\eqref{SWS}
(Theorem~1  in~\cite{AN}) and all generic metrics  with positive Ricci curvature will be deformed into metrics with mixed Ricci curvature for  $W_{12}$ and $W_{24}$ (see Theorem~2 in~\cite{AN}), where given a metric is said to be generic if $x_i \ne x_j\ne x_k\ne x_i$ for $i,j,k\in \{1,2,3\}$.
According to Theorems 3 and 4 in~\cite{AN} and Theorem 3 in~\cite{Ab7}
positiveness of the Ricci curvature will be preserved
for all metrics at $a\in (1/6,1/2)$
and for a special kind of metrics
satisfying $x_k<x_i+x_j$ at $a=1/6$
(the equalities $x_k=x_i+x_j$ correspond to  K\"{a}hler metrics),
whereas all positively curved metrics will be deformed into metrics with mixed  Ricci curvature
if $a\in (0,1/6)$.
In~\cite{Ab7, AN} we used the description
\begin{equation}\label{vali}
\gamma_i:=(x_j-x_k)^2+2x_i(x_j+x_k)-3x_i^2, \quad \{i,j,k\}=\{1,2,3\},
\end{equation}
of the set of Riemannian metrics with positive sectional curvature on the Wallach spaces~\eqref{SWS}  given in~\cite{Valiev}
and the description
\begin{equation}\label{ricci}
\lambda_i:=x_jx_k+a\,\big(x_i^2-x_j^2-x_k^2\big), \quad \{i,j,k\}=\{1,2,3\},
\end{equation}
of the set of Riemannian metrics of positive Ricci curvature
on generalized Wallach spaces in the case  $a_1=a_2=a_3:=a$
given in~\cite{Nikonorov2}.
According to~\cite{Valiev} the system of inequalities $\gamma_i>0$  describes the set~$S$
of Riemannian metrics of positive sectional curvature on the Wallach spaces~\eqref{SWS}
and, analogously, solutions of the system of inequalities
$\lambda_i>0$ represents the set~$R$ of all Riemannian metrics of positive Ricci curvature
on every generalized Wallach spaces with $a\in (0,1/2)$ as follows from~\eqref{Ricci_princ}.

In~\cite{Ab7, AN} we relied on Maple evaluations and our visual observations concerning
structural properties of  surfaces and curves obtained from  \eqref{vali} and \eqref{ricci},
but  justifications of those properties were not included into the text of the papers.
Filling  these gaps  we initiated in~\cite{Ab24}.
The present paper continues that idea.
For $i=1,2,3$ denote by $\Gamma_i$ and $\Lambda_i$ the surfaces
in~$(0,+\infty)^3$
defined by the equations $\gamma_i=0$ and $\lambda_i=0$
respectively and introduce  space curves  $s_i:=\Sigma\cap \Gamma_i$,
$r_i:=\Sigma\cap \Lambda_i$ and sets $\Sigma S:=\Sigma \cap S$, $\Sigma R:=\Sigma \cap R$.
The main result of this paper is contained in the following two theorems.

\begin{theorem}\label{theo24_1}
The following assertions hold
for all indices with $\{i,j,k\}=\{1,2,3\}$:
\begin{enumerate}
\item
For each Wallach space in~\eqref{SWS}  the set $\Sigma S$
of invariant Riemannian metrics~\eqref{metric} which belong to the invariant set~$\Sigma$ of the system~\eqref{three_equat} and
admit positive sectional curvature
is  bounded by the pairwise disjoint regular space curves
$s_1$, $s_2$ and $s_3$ in~$\Sigma$ such that each~$s_k$ is connected and
can be parameterized  as the following
\begin{eqnarray*}\label{param_s2024}
x_k=t^{-1}\alpha^{-2}, \quad x_i= t\alpha, \quad x_j=\alpha,
\end{eqnarray*}
where
$\alpha=\alpha(t):=\left(\dfrac{-t-1+2\sqrt{t^2-t+1}}
{t(t-1)^2}\right)^{\frac{1}{3}}>0$ ~and~ $t\in (0,1)\cup (1,+\infty)$;
\item
Every invariant curve $I_k$ of the system~\eqref{three_equat}
given by the equations $x_i=x_j=p$, $x_k=p^{-2}$, $p>0$,
intersects the only border curve  $s_k$ at the unique
point with coordinates $x_i=x_j=p_0$, $x_k=p_0^{-2}$
approaching at infinity the other two curves~$s_i$ and $s_j$
as close as we like,  where $p_0=\sqrt[3]{6}/2$.
\end{enumerate}
\end{theorem}

The results of Theorem~\ref{theo24_1}
are illustrated in the left panel of Figure~\ref{s_curves},
where the curves $s_1$, $s_2$ and $s_3$ are depicted respectively
in red, teal and blue colors,
the invariant curves $I_1,I_2, I_3$ are all  yellow colored.

\begin{figure}[h]
\centering
\includegraphics[angle=0, width=0.45\textwidth]{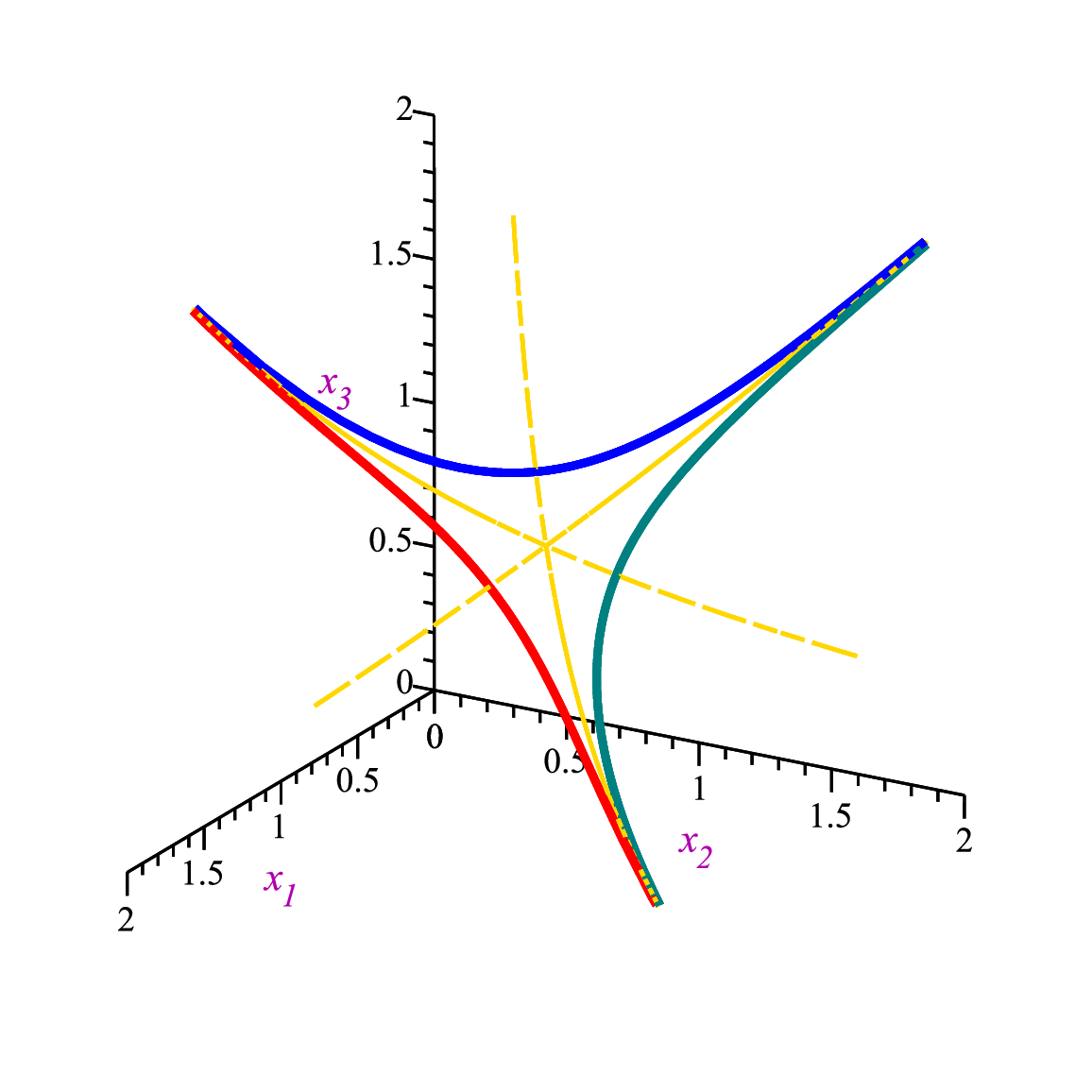}
\includegraphics[angle=0, width=0.45\textwidth]{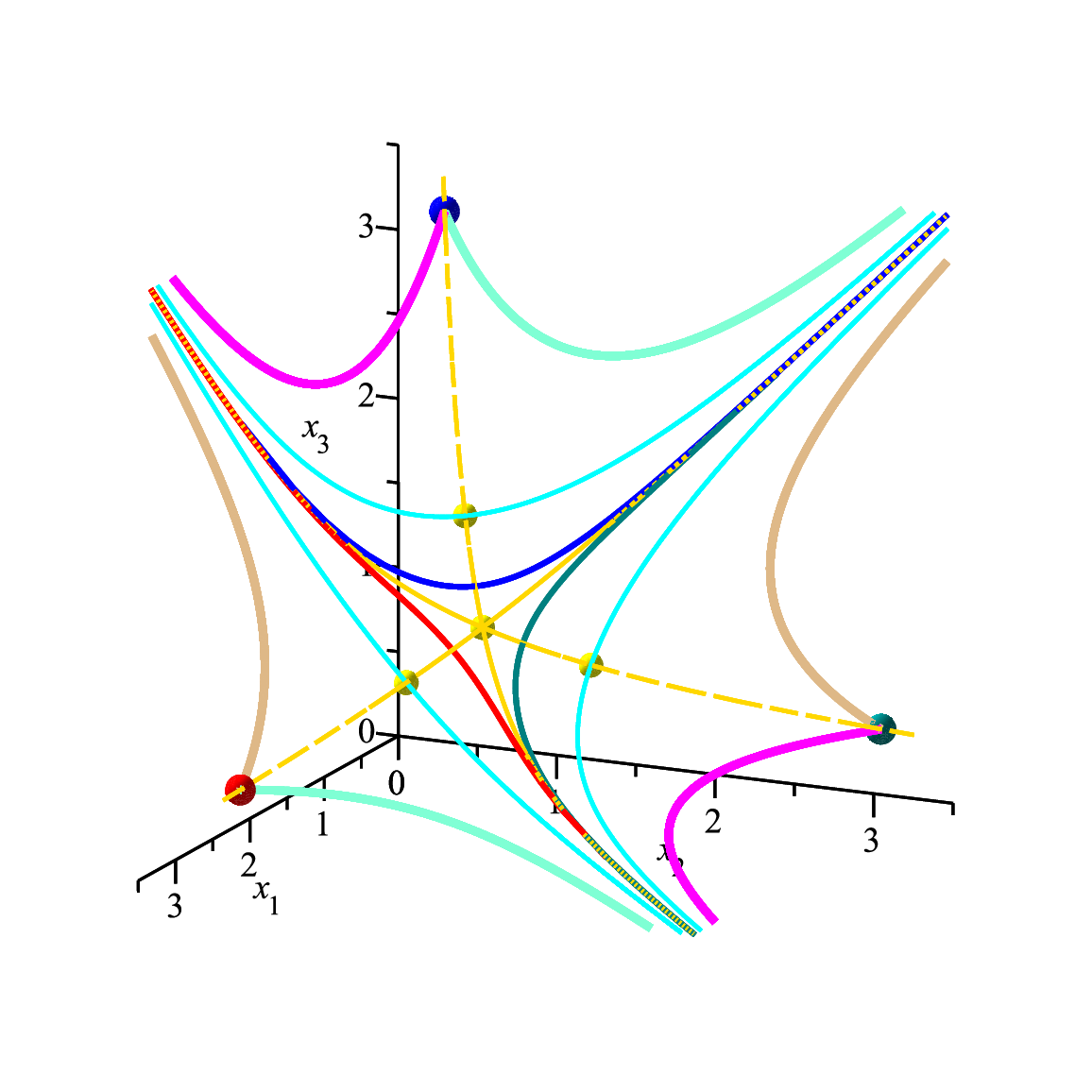}
\caption{The left panel: the  curves $s_1, s_2, s_3$; The right panel: the curves  $r_1, r_2, r_3$, $l_1, l_2, l_3$ and singular points  $\bold{o}_0, \bold{o}_1, \bold{o}_2,\bold{o}_3$  corresponding to $a=1/6$.}
\label{s_curves}
\end{figure}

\begin{theorem}\label{theo24_2}
The following assertions hold
for all indices with $\{i,j,k\}=\{1,2,3\}$:
\begin{enumerate}
\item
For every  generalized  Wallach space with $a_1=a_2=a_3=a\in (0,1/2)$
the set $\Sigma R$
of invariant Riemannian metrics~\eqref{metric} which belong to the invariant set~$\Sigma$ of the system~\eqref{three_equat} and
admit positive Ricci curvature
is  bounded by the space curves
$r_1$, $r_2$ and $r_3$ in~$\Sigma$ such that
each~$r_k$ consists of two  regular connected components $r_{ki}$ and $r_{kj}$
parameterized by equations
\begin{equation}\label{param_low}
x_k=t^{-1}\beta^{-2}, \quad x_i= t\beta, \quad x_j=\beta
\end{equation}
and
\begin{equation}\label{param_up}
x_k=t^{-1}\beta^{-2}, \quad x_j= t\beta, \quad x_i=\beta
\end{equation}
respectively, where
$\beta=\beta(t):=\left(t^4-a^{-1}t^3+t^2\right)^{-\frac{1}{6}}>0$
~and~ $t\in (0, a]$.
\item
Every pair of the curves $r_i$ and $r_j$  admits a unique common point $P_{ij}$ with  coordinates
$x_i=x_j=a^{\frac{1}{3}}$, $x_k=a^{-\frac{2}{3}}$
which belong to the components $r_{ij}$ and $r_{ji}$;
In addition, every invariant curve  $I_k$  of the system~\eqref{three_equat}
meets  the components $r_{ij}$ and $r_{ji}$ of $r_{i}$ and~$r_{j}$ exactly at the point $P_{ij}$
approaching their another components $r_{ik}$ and $r_{jk}$
at infinity as close as we like.
\item
For every  $a\in (0,1/2)$
all singular (equilibria) points
of the system~\eqref{three_equat}
belong to the set $\Sigma R$.
\item
 K\"ahler metrics $x_k=x_i+x_j$ of generalized Wallach spaces with  $a=1/6$  form in $\Sigma$ separatrices $l_k$ of saddles of~\eqref{three_equat} which  can be defined by parametric equations
\begin{equation}\label{kaler_curve}
x_k=t^{-1}\phi^{-2}, \quad x_i=t\phi, \quad x_j= \phi,
\end{equation}
where $\phi=\phi(t):=(t^2+t)^{-1/3}$ ~and~ $t>0$.
\end{enumerate}
\end{theorem}

The results of Theorem~\ref{theo24_2}
are illustrated in the right panel of Figure~\ref{s_curves}
for the case $a=1/6$,
where the curves $r_1$, $r_2$ and $r_3$ are depicted respectively
in magenta, aquamarine and burlywood colors,
the curves  $l_1$, $l_2$ and $l_3$ are depicted by cyan colored curves
and  yellow colored points correspond to singular  points of~\eqref{three_equat}.

\bigskip

It should be noted that we will consider only Riemannian metrics
satisfying the unit volume condition $x_1x_2x_3=1$  (see~\cite{AANS1, AN}),
and denote by~$\Sigma$ the  surface defined by equation $V\equiv 1$ with
$V:=x_1x_2x_3$.
In general, surfaces $x_1x_2x_3=c$, where $c>0$,
play the significant role for study~\eqref{three_equat}
on generalized Wallach spaces.
It is known that any set determined by the
equation $x_1x_2x_3=c$
is invariant under~\eqref{three_equat}, moreover $x_1x_2x_3=c$ is its first integral.
Surfaces $x_1x_2x_3=c$ will  also be unstable (or stable) manifolds of~\eqref{three_equat} and contain leading directions of motions of its trajectories (see~\cite{Ab24}).
Since the right hand sides of~\eqref{three_equat}
are all homogeneous, namely
$f_i(cx_1,cx_2,cx_3)=f_i(x_1,x_2,x_3)$ for any $c$,
 we can  pass to a new differential system of the same form as the original one, but
with $\widetilde{x}_1\widetilde{x}_2\widetilde{x}_3=1$.
Actually this is reachable by replacings\, $x_i(t)=\widetilde{x}_i(\tau)\sqrt[3]{c}$ and $t=\tau\sqrt[3]{c}$.
Therefore  without loss of generality
we assume that  the invariant surface  is given by~$V\equiv 1$.

\section{Auxiliary results}

Observe that the expressions for $\gamma_i$ and $\lambda_i$ in \eqref{vali} and \eqref{ricci} are symmetric under the permutations $i\to j\to k\to i$.
Therefore it suffices to consider representatives only at fixed $(i,j,k)$,
where $\{i,j,k\}=\{1,2,3\}$.

\begin{lemma}[\cite{Ab24}]\label{Gamma_Conic}
For each Wallach space in~\eqref{SWS}
the set  $S$ of Riemannian metrics \eqref{metric} with positive sectional curvature is bounded by the pairwise disjoint conic surfaces
 $\Gamma_1$, $\Gamma_2$ and $\Gamma_3$.
\end{lemma}

The cones $\Gamma_1$, $\Gamma_2$ and $\Gamma_3$ are depicted
in the left panel of Figure~\ref{surfaces} in red, teal and blue colors respectively.

Actually Lemma~\ref{Gamma_Conic} was proved in~\cite{Ab24}.
Here we bring the sketch of its proof for convenience of the readers.
The equation  $\gamma_k=0$ defines two connected components
$3x_k= x_i+x_j- 2\sqrt{x_i^2-x_ix_j+x_j^2}$
and
$$
x_k= \Phi_k(x_i,x_j):=3^{-1}\left(x_i+x_j+ 2\sqrt{x_i^2-x_ix_j+x_j^2}\right)
$$
of the cone $\Gamma_k$.
Since the first of them gives  $x_k<0$ for all  $x_i, x_j>0$
then $\gamma_k>0$ is equivalent to
$0<x_k<\Phi_k(x_i,x_j)$
meaning that~$S$ is bounded by the plane  $x_k=0$ and the  positive part~$\Gamma_k$
of the cone~$\gamma_k=0$. By symmetry we have the same
for $\Gamma_i$ and $\Gamma_j$.

Consider the pair $(i,j)$. The equations $\gamma_i=0$ and $\gamma_j=0$
defining the surfaces~$\Gamma_i$ and $\Gamma_j$
can admit only the following two family of common solutions $x_i=x_j$, $x_k=0$
and $x_i=x_k$, $x_j=0$. But we need in positive solutions only.
Hence $\Gamma_i \cap \Gamma_j=\emptyset$ for all  positive $x_1, x_2, x_3$.
By symmetry the same  assertions hold for the pairs $(i,k)$ and $(j,k)$.

\begin{figure}[h]
\centering
\includegraphics[angle=0, width=0.45\textwidth]{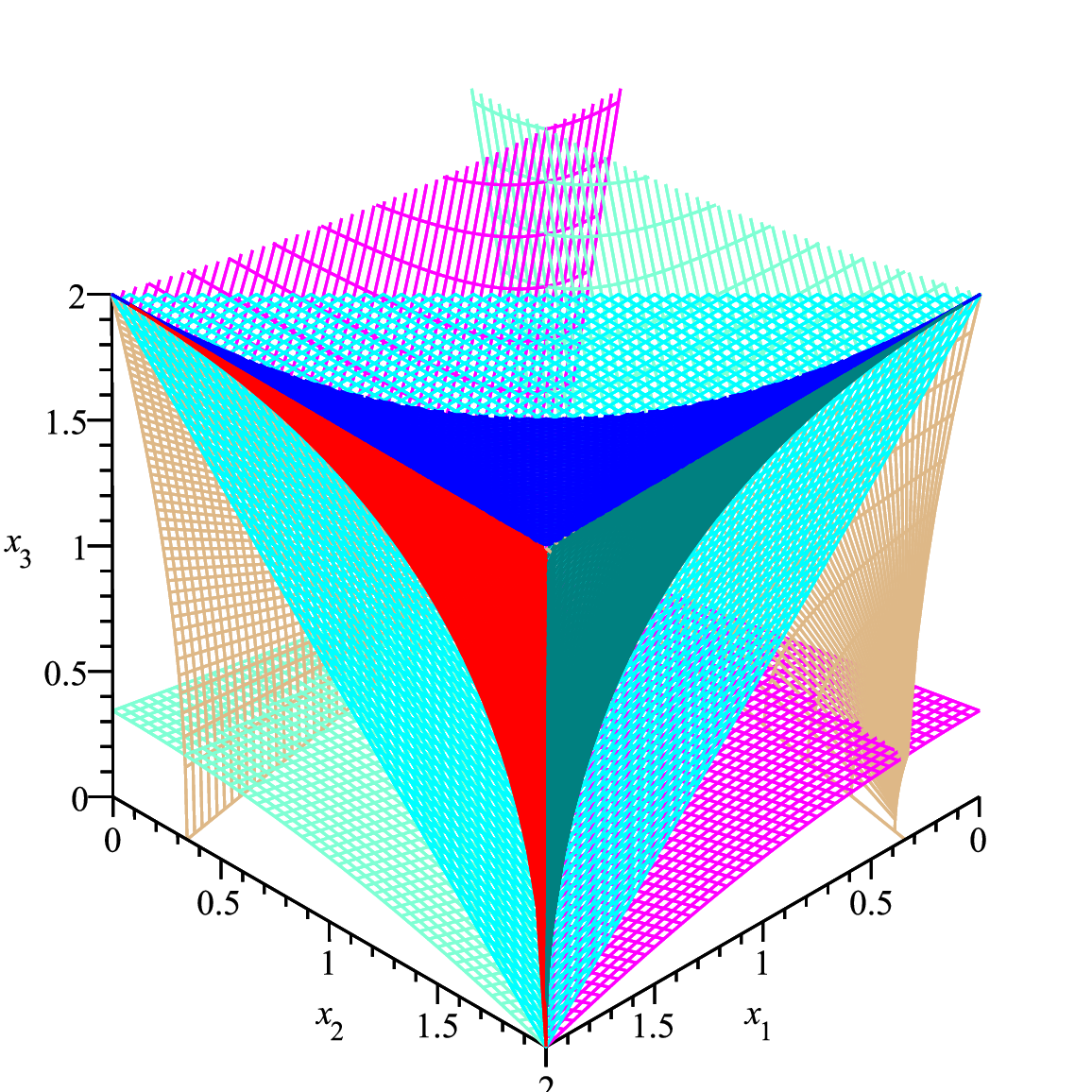}
\includegraphics[angle=0, width=0.45\textwidth]{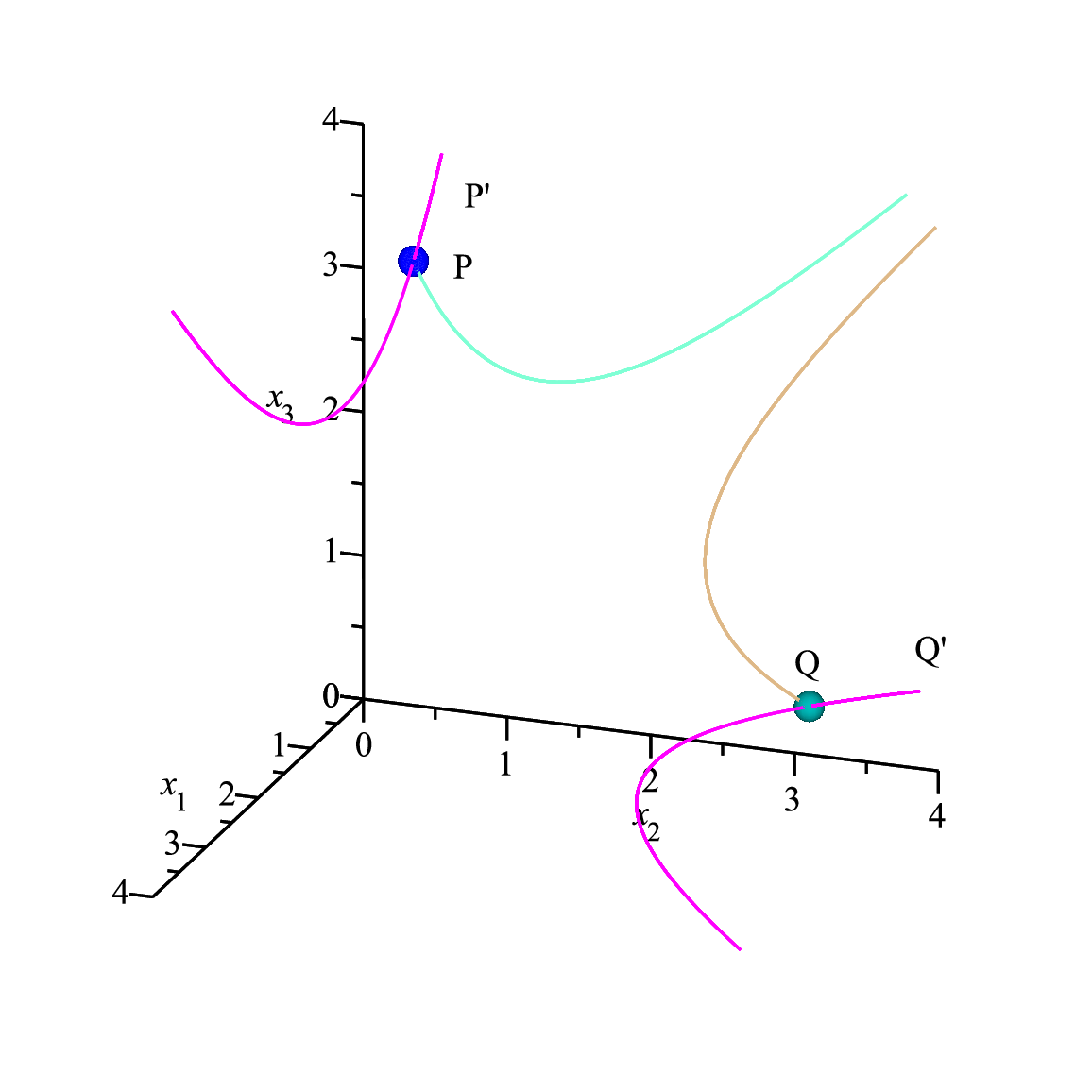}
\caption{The left panel: the cones $\Gamma_1,\Gamma_2, \Gamma_3$,
$\Lambda_1, \Lambda_2, \Lambda_3$ and the planes $x_k=x_i+x_j$ for $\{i,j,k\}=\{1,2,3\}$;
The right panel: crossing $r_2$ and $r_3$ by $r_1$.}
\label{surfaces}
\end{figure}

\begin{lemma}\label{Delta_Conic}
For every  generalized  Wallach space with $a\in (0,1/2)$ the set $R$
is bounded by the  conic surfaces
$\Lambda_1$, $\Lambda_2$ and $\Lambda_3$.
Each pair $\Lambda_i$ and $\Lambda_j$  has intersections along two different straight lines  $x_i=x_j=u$, $x_k=0$ and $x_i=x_j=av$, $x_k=v$, where $u,v>0$.
\end{lemma}

The cones $\Lambda_1$, $\Lambda_2$ and $\Lambda_3$ are depicted
in the left panel of Figure~\ref{surfaces} in  magenta, aquamarine and burlywood colors respectively.

\medskip

\begin{proof}[Proof of Lemma~\ref{Delta_Conic}]
Consider the surface $\Lambda_k$.
Since $D:=x_i^2-a^{-1}x_ix_j+x_j^2$ is symmetric with respect to $x_i$ and $x_j$
it  can be considered as a quadratic polynomial in~$x_j$
without loss of generality.
Then
$D\le 0$ if $mx_i\le x_j\le Mx_i$ and
$D>0$ if $0<x_j<mx_i$ or $x_j>Mx_i$,
where
\begin{equation}\label{m_and_M}
m:=\left(1-\sqrt{1-4a^2}\right)(2a)^{-1}>0, \quad
M:=\left(1+\sqrt{1-4a^2}\right)(2a)^{-1}>0.
\end{equation}
Depending on the sign  of~$D$
the inequality  $\lambda_k>0$  admits the positive solution
$x_k> \sqrt{D}$
if $D>0$ and any $x_k>0$ can satisfy  $\lambda_k>0$ if $D\le 0$.
This means that besides the planes $x_1=0$, $x_2=0$  and $x_3=0$ the set~$R$
is bounded by two disjoint connected components $\Lambda_{kj}$
and $\Lambda_{ki}$ of the surface $\Lambda_k=\Lambda_{ki}\cup \Lambda_{kj}$
defined by the same equation
$$
x_k=\Psi(x_i,x_j):=\sqrt{x_i^2-a^{-1}x_ix_j+x_j^2}
$$
but on different domains
$\left\{(x_i,x_j)\in \mathbb{R}^2~ \big| ~ x_i>0, \, 0<x_j<mx_i\right\}$
and
$\left\{(x_i,x_j)\in \mathbb{R}^2~ \big| ~ x_i>0, \, x_j>Mx_i\right\}$
respectively.

Due to  symmetry the same properties hold for
the surfaces $\Lambda_i$ and $\Lambda_j$ as well.
Thus $\partial (R)=\Lambda_1\cup \Lambda_2 \cup \Lambda_3$.

By the same reason of symmetry it suffices to analyze $\Lambda_i \cap \Lambda_j$ only.
Assume that some triple $(x_1,x_2,x_3)$ satisfies both of
 $\lambda_i=0$ and $\lambda_j=0$.
Then $\lambda_i-\lambda_j=0$ and $\lambda_i+\lambda_j=0$ imply
$(x_i-x_j)\big(x_k-2a(x_i+x_j)\big)=0$ and $x_k(x_i+x_j-2ax_k)=0$.
In what follows that the system of equations $\lambda_i=0$ and $\lambda_j=0$
can admit only the following two different families of one-parametric solutions $x_i=x_j=u$, $x_k=0$ and
\begin{equation}\label{common_lines}
x_i=x_j=av, \quad x_k=v
\end{equation}
with  parameters  $u,v>0$.
\end{proof}

\section{Proofs of the main results}

\begin{proof}[Proof of Theorem \ref{theo24_1}]
$(1)$
{\it The curves $s_1, s_2, s_3$ are pairwise disjoint and
 bound the set $\Sigma S$}\/ because so are
the corresponding cones  $\Gamma_1, \Gamma_2, \Gamma_3$
and $\partial(S)=\Gamma_1\cup \Gamma_2\cup \Gamma_3$
according to Lemma~\ref{Gamma_Conic}.

\smallskip

{\it Parameterizations of the curves $s_1$, $s_2$ and $s_3$.}
Due to symmetry fix any unordered triple $(i,j,k)$.
Putting $x_k=x_i^{-1}x_j^{-1}$ in $\gamma_k=0$
we obtain the following  polynomial equation of degree $6$ in two variables $x_i$ and $x_j$:
 $$
 x_i^4x_j^2-2x_i^3x_j^3+x_i^2x_j^4+2x_i^2x_j+2x_ix_j^2-3=0.
 $$
 After then substituting  $x_i=tx_j$,  $x_j=\sqrt[3]{u}$ into the obtained equation and solving it  with respect to~$u$ we can get a parametric
representation of the curve $s_k$ given in~\eqref{param_s2024}.
It is expected its two different roots
$u_{1,2}:=\left(-t-1 \pm 2\sqrt{t^2-t+1}\right)
t^{-1}(t-1)^{-2}$, $t>0$, $t\ne 1$,
but the second of these roots taken with the minus sign gives only negative values
of~$x_i$ and $x_j$.
Denote  $\widetilde{\alpha}(t)=\sqrt[3]{u_1(t)}>0$.

Note also that $\lim_{t\to 0+}\widetilde{\alpha}(t)=+\infty$ and
$\lim_{t\to +\infty}\widetilde{\alpha}(t)=0$.
This yields
$\lim_{t\to +\infty}x_j(t)=0$, $\lim_{t\to +\infty}x_i(t)=\lim_{t\to +\infty}x_k(t)=+\infty$
and $\lim_{t\to 0+}x_i(t)=0$, $\lim_{t\to 0+}x_j(t)=\lim_{t\to 0+}x_k(t)=+\infty$.

\smallskip

{\it Connectedness of the curves $s_1$, $s_2$ and $s_3$.}
Note that $\lim_{t\to 1+}\widetilde{\alpha}(t)=\lim_{t\to 1-}\widetilde{\alpha}(t)=p_0:=\sqrt[3]{6}/2$.
Hence assigning $\alpha(1)=p_0$  we define on $G:=(0,+\infty)$ a continuous function
$\alpha \colon G\rightarrow G$
$$
\alpha(t)=
 \begin{cases}
\widetilde{\alpha}(t), &\mbox{if~~} t>0, ~t\ne 1,\\
p_0, &\mbox{if~~} t=1.
\end{cases}
$$

Therefore in the standard topology of $\mathbb{R}^3$ the set (curve)
$s_k=F(G)\subset G^3$ must be connected  as a continuous image of the connected set $G$
under a function  $F\colon G\rightarrow G^3$
with  continuous coordinate components  $x_i, x_j, x_k\colon G\rightarrow G$ such that
$x_i(t)=t\alpha (t)$, $x_j(t)=\alpha(t)$ and $x_k(t)=t^{-1}\alpha(t)^{-2}$.

\smallskip

{\it Smoothness of the curves $s_1$, $s_2$ and $s_3$}
can be proved using~\eqref{param_s2024}.
But we prefer another  way. Due to symmetry it suffices
to prove smoothness of the curve $s_i=\Sigma\cap \Gamma_i$.
Since  $\Sigma$ and $\Gamma_i$ are smooth (regular) surfaces
it remains to show that their intersection is transversal, in other words
their gradient vectors
$\nabla V=(x_2x_3,x_1x_3,x_1x_2)=\big(x_1^{-1}, x_2^{-1}, x_3^{-1}\big)$ and
$\nabla \gamma_i=\left(\gamma_{i1}, \gamma_{i2}, \gamma_{i3}\right)$ are linearly independent
along $s_i$, where
$$
\gamma_{ij}:=\dfrac{\partial \gamma_i}{\partial x_j}
=\begin{cases}
x_i+x_j-x_k, &\mbox{if~~} j\ne i,\\
-3x_{i}+x_j+x_k, &\mbox{if~~} j=i,
\end{cases}
$$
for $i,j\in \{1,2,3\}$.
Due to symmetry fix any $i$ and suppose by contrary that $\nabla \gamma_i=c\, \nabla V$ for some  real $c\ne 0$.
This means that the equalities  $\gamma_{ij}=c x_j^{-1}$
hold for  $j\in {1,2,3}$.
Then for $j\ne i$ and $k\ne i$ we obtain
 equalities
$$
(x_i+x_j-x_k)\, x_j=(x_i+x_k-x_j)\, x_k=c
$$
equivalent to $(x_j-x_k)(x_i+x_j+x_k)=0$ which is impossible
for  $x_i\ne x_j\ne x_k\ne x_i$.
Actually we proved the more strong fact that  the normal vectors
$\nabla V$ and $\nabla \gamma_i$ are linearly independent
not only along $s_i$, but everywhere where the surfaces
$\Sigma$ and  $\Gamma_i$ are defined excepting points $(x_1,x_2,x_3)$
with non positive or coincided components.

\smallskip

$(2)$  Due to symmetry  it suffices to take the  invariant curve $I_k$ of~\eqref{ricciflow}
defined as $x_i=x_j=p$, $x_k=p^{-2}$, $p>0$.
Consider the curve $s_k$.
The question is will $I_k$ cross the curve  $s_k$ or not.
It suffices to answer this question for~$I_k$ and the surface $\Gamma_k$
because existing of a point $Z$ in $(0,+\infty)^3$ such that  $Z\in I_k\cap \Gamma_k$
implies  $Z\in I_k\subset \Sigma$ and hence   $Z\in \Sigma \cap \Gamma_k=s_k$.
Thus substituting $x_i=x_j=p$, $x_k=p^{-2}$ into the equation $\gamma_k=0$ of $\Gamma_k$
we obtain the equation
$\big(4p^3-3\big)p^{-4}=0$  which admits the single root $p=p_0=\sqrt[3]{6}/2$ providing the unique common point $x_i=x_j=p_0$, $x_k=p_0^{-2}$ of~$I_k$ with~$s_k$.

Consider now any curve $s_i$ such that $i\ne k$.
Then we obtain an incompatible system
of equations $x_i=x_j=p$, $x_k=p^{-2}$ and $\gamma_i=0$ because of
$\gamma_i=p^{-4}\ne 0$.
Moreover, $s_i$ asymptotically tends to~$I_k$ as $p\to +\infty$ according  to
$\lim_{p\to +\infty}\gamma_i=\lim_{p\to +\infty}p^{-4}=0$.
The same result holds for $s_j$ by symmetry in the equation of~$I_k$.
Theorem \ref{theo24_1} is proved.
\end{proof}

\bigskip

\begin{proof}[Proof of Theorem \ref{theo24_2}]
$(1)$ Clearly $\partial(\Sigma R)=r_1\cup r_2\cup r_3$ directly follows
from $\partial(R)=\Lambda_1\cup \Lambda_2\cup \Lambda_3$ proved in
Lemma~\ref{Delta_Conic}.
Intersecting both of the connected components
$\Lambda_{ki}$ and $\Lambda_{kj}$ of the cone  $\Lambda_k$
the surface~$\Sigma$ forms  components $r_{ki}=\Sigma \cap \Lambda_{ki}$ and
$r_{kj}=\Sigma \cap \Lambda_{kj}$
of the curve $r_k$ such that $r_k=r_{ki} \cup r_{kj}$ and
  $r_{ki} \cap r_{kj}=\emptyset$.

\smallskip

{\it Smoothness of the components of  $r_1$, $r_2$ and $r_3$}.
Consider the curve $r_k$.
We claim that the
gradient vectors
$\nabla V=\big(x_1^{-1}, x_2^{-1}, x_3^{-1}\big)$ and
$\nabla \lambda_k=\left(\lambda_{k1}, \lambda_{k2}, \lambda_{k3}\right)$
of the surfaces $\Sigma$ and $\Lambda_k$ are linearly independent
for all positive  $x_1,x_2,x_3$ such that $x_1\ne x_2\ne x_3\ne x_1$,
where
$$
\lambda_{kj}:=\dfrac{\partial \lambda_k}{\partial x_j}
=\begin{cases}
x_i-2ax_j, &\mbox{if~~} j\ne k,\\
2ax_{k}, &\mbox{if~~} j=k,
\end{cases}
$$
for $k,j\in \{1,2,3\}$.
Indeed supposing $\nabla \lambda_k=c\, \nabla V$, where $c$ is a nonzero real number,
we obtain immediately an unreachable equality $(x_j-x_i)(x_j+x_i)=0$.
In what follows that each component~$r_{k1}$  and $r_{k2}$
of the curve~$r_k$ is  a smooth curve as a transversal intersection of two smooth surfaces.

\smallskip

{\it Connectedness of the components of $r_1, r_2, r_3$.}
The variable $x_k$ can be eliminated  from the system of equations $x_ix_jx_k=1$ and $\lambda_k=0$
to obtain the equation
$$
a\left(x_i^4x_j^2+x_i^2x_j^4\right)-x_i^3x_j^3-a=0
$$
of the projection of the curve $r_k$
onto the coordinate plane $(x_i,x_j)$.
By the same way as in Theorem~\ref{theo24_1}
substituting $x_i=tx_j$,  $x_j=\sqrt[3]{u}$ into the last equation
and solving it with respect to~$u$ we obtain a parametric
equation
$$
x_k=t^{-1}\beta^{-2}, \quad x_i= t\beta, \quad x_j=\beta
$$
of the curve $r_k$,
where
$$
\beta=\beta(t):=\left(t^2(t-m)(t-M)\right)^{-\frac{1}{6}}.
$$

The function $\beta(t)$ is
well defined, continuous and positive valued for $t\in (0,m)$ or $t\in (M,+\infty)$,
where $m$ and $M$ are given in~\eqref{m_and_M}, $0<m<M$.
It follows then the components $r_{ki}$ and $r_{kj}$  of~$r_k$ are
respectively continuous  images of the connected sets $(0,m)$
and $(M,+\infty)$  under a vector-function  with
 coordinates $x_i(t)$, $x_j(t)$ and $x_k(t)$.
Therefore $r_{ki}$ and $r_{kj}$ are connected too.

Note  that the components  $r_{ki}$ and $r_{kj}$
are symmetric under the permutation $i\to j\to i$.
Therefore we can parameterize them on the same interval
but using different formulas
\eqref{param_low} and \eqref{param_up} respectively.
For simplicity we choose the interval $(0,m)$.

 \smallskip

{\it Intersections of $r_1$, $r_2$ and $r_3$}.
Consider the pair  $r_1$ and $r_2$.
By Lemma~\ref{Delta_Conic}  the surfaces $\Lambda_1$ and $\Lambda_2$  meet each other along
the straight line $x_1=x_2=av$, $x_3=v$ (see~\eqref{common_lines}).
This line intersects the surface $\Sigma$
at a unique point $P_{12}$. Indeed substituting
$x_1=x_2=av$, $x_3=v$ into $x_1x_2x_3=1$ we get the unique value
$v=v_0:=a^{-2/3}$.
This yields coordinates $\big(a^{1/3},a^{1/3}, a^{-2/3}\big)$  of $P_{12}$.
Note that $P_{12}$ (the point $P$ in the right panel of Figure~\ref{surfaces})
is also the only intersection point of the curves $r_1$ and $r_2$
(their components $r_{12}$ and $r_{21}$).

Now a value of $t$ at which  $P_{12}$  belongs to $r_1$
can be found from the parametric representation
$$
x_1(t)=t^{-1}\beta(t)^{-2},  \quad x_2(t)=t\beta(t), \quad x_3(t)= \beta(t)
$$
of $r_{12}$.
The condition $x_1=x_2$ implies an equation $t^{-1}\beta^{-2}= t\beta$
which has the single root $t_0=a$ for all $a\in (0,1/2)$.
Therefore the curve  $r_1$  passes through $P_{12}$ at $t=t_0$ only.
It should be noted that  the curves $r_1$  and $r_2$ leave extra pieces after crossing each other.
In principle, we can preserve them, but
it is advisable to remove them for greater clarity of pictures.
Analyzing the values of limits
$$
\lim_{t\to 0+}x_2(t)=0,~~ \lim_{t\to 0+}x_1(t)=\lim_{t\to 0+}x_3(t)=+\infty
$$
and
$$
\lim_{t\to m-}x_1(t)=0, ~~
\lim_{t\to m-}x_2(t)=\lim_{t\to m-}x_3(t)=+\infty
$$
we conclude that the tail $PP'$ corresponds to values $t\in (a,m)$.
Therefore  the original interval of parametrization $(0,m)$ can be reduced
to the interval  $(0,a]$ shown in the text of Theorem~\ref{theo24_2}.

By symmetry the analysis of the  pairs $r_1 \cap r_3$ and $r_2\cap r_3$
(points in teal and red color in the right panel of Figure~\ref{s_curves})
will be the same using permutations of the indices $\{i, j, k\}=\{1,2,3\}$.
For example, the equations
$x_1(t)=t^{-1}\beta(t)^{-2}$,  $x_2(t)=\beta(t)$ and $x_3(t)= t\beta(t)$
define another connected component $r_{13}$  of the curve $r_1$ (which intersects~$r_3$)
on the same interval $(0,a]$. Then coordinates
$\big(a^{1/3},a^{-2/3}, a^{1/3}\big)$ of the point $P_{13}$
(in fact $\{P_{13}\}=r_{13} \cap r_{31}$)
can be obtained at the same boundary value $t=a$ (the point~$Q$ in the right panel of Figure~\ref{surfaces}).
Analogously at $t\in (a,m)$ we get  the part $QQ'$  of $r_{13}$.

\smallskip

$(2)$
Without loss of generality consider the  invariant curve $I_k$.
As noted in Theorem~\ref{theo24_1}
it suffices to consider the surfaces $\Lambda_i$
instead of the corresponding curves $r_i$.
The curve $I_k$ crosses both of the curves $r_i$ and $r_j$
(the components~$r_{ij}$ and $r_{ji}$)
exactly at their common point $P_{ij}$
because substituting $x_i=x_j=p$,  $x_k=p^{-2}$ into $\lambda_i=0$ and $\lambda_j=0$
yields the equations
$$
\lambda_i=\lambda_j=\big(p^3-a\big)p^{-4}=0
$$
which admit a single root $p=a^{1/3}$ corresponding to $P_{ij}$.
Therefore $I_k\cap r_{ij}\cap r_{ji}=\{P_{ij}\}$.

At the same time $I_k$
approximates both of $r_i$  and $r_j$ (their components~$r_{ik}$ and $r_{jk}$) at infinity.
Indeed
$$
\lim_{p\to +\infty}\lambda_i=\lim_{p\to +\infty}\lambda_j=
\lim_{p\to +\infty}\big(p^3-a\big)p^{-4}=0.
$$

For the curve $r_{k}$ we have
$$
\lambda_k=(1-2a)p^2+p^{-4}> 0
$$
under the same substitutions.
Therefore $I_k$ never cross $r_k$, moreover, $\lim_{p\to +\infty}\lambda_k=+\infty$.

\smallskip

$(3)$
As it follows from \cite{AANS1, Lomshakov2}
the system of algebraic equations
$f_i(x_1,x_2,x_3)=0$
has the following four families of one-parametric solutions
for every $a\in (0,1/2)\setminus \{1/4\}$:
\begin{equation}\label{lines_equilib}
x_1= x_2= x_3=\tau,   \qquad
x_i=\tau\kappa, ~ x_j=x_k=\tau,   \quad \tau>0, \quad  \{i,j,k\}=\{1,2,3\},
\end{equation}
 where
\begin{equation}\label{kappa}
\kappa:=\frac{1-2a}{2a}.
\end{equation}
At $a=1/4$ these families merge to the unique family $x_1=x_2=x_3=\tau$.

Substituting  $x_i=\tau\kappa$ and $x_j=x_k=\tau$ into
the expressions \eqref{ricci} for $\lambda_1, \lambda_2$ and $\lambda_3$ we obtain
$$
\lambda_1=\lambda_2=\lambda_3=\frac{1-4a^2}{4a}\,\tau^2>0
$$
because $a\in (0,1/2)$.
Obviously,
$$
\lambda_1=\lambda_2=\lambda_3=(1-a)\,\tau^2>0
$$
at $x_1= x_2= x_3=\tau$.
Therefore the straight lines  \eqref{lines_equilib} lye in the set $R$
for all $a\in (0,1/2)$. They cross the invariant surface~$\Sigma$
at the points (see also~\cite{Ab24})
$$
\bold{o}_0:= (1, 1, 1), \quad \bold{o}_1:=\left(q\kappa, q, q\right),
\quad \bold{o}_2:= \left(q, q\kappa, q\right), \quad
\bold{o}_3:=\left(q, q, q \kappa\right),
$$
being the singular points of the system \eqref{three_equat} on $\Sigma$,  where
$q:=\sqrt[3]{\kappa^{-1}}$
(at $a=1/4$ the unique singular point $(1,1,1)$ is obtained
according to \eqref{kappa}).
Thus we conclude that  $\bold{o}_i\in \Sigma R$ for every
$a\in (0,1/2)$ and $i\in \{0,1,2,3\}$.

\smallskip

$(4)$
According to~\cite{Ab7, AANS1, AANS2}
the curves $I_1$, $I_2$ and $I_3$ are
separatrices of the unique saddle point  $\bold{o}_0$  (of the {\it linear zero type})
of the system \eqref{three_equat} if $a=1/4$.
For $a\in (0,1/2)\setminus \{1/4\}$
the points $\bold{o}_1, \bold{o}_2, \bold{o}_3$ are all  {\it hyperbolic type}\/ saddles
and $\bold{o}_0$ is a stable (respectively unstable) {\it hyperbolic}\/ node
if $1/4<a<1/2$ (respectively $0<a<1/4$).
Additionally, each invariant curve $I_k$  is one of two
separatrices of the  saddle $\bold{o}_k$ (see~\cite{Ab24}).
At $a=1/6$ we have an opportunity
to find analytically the second separatrice of each $\bold{o}_k$
different from $I_k$, where  $k=1,2,3$.
Indeed it is easy to see that coordinates of~$\bold{o}_k$
satisfy the system of equations
\begin{equation}\label{kaler_sep}
\aligned
\begin{cases}
x_k=x_i+x_j, \\
x_1x_2x_3=1,
\end{cases}
\endaligned
\end{equation}
where  $x_k=x_i+x_j$
represents  K\"{a}hler metrics
on a given generalized Wallach space $G/H$  with $a_1=a_2=a_3=a=1/6$,
$\{i,j,k\}=\{1,2,3\}$  (see also~\cite{AN}).
Therefore each saddle~$\bold{o}_k$
belongs to a curve~$l_k$ obtained  as an intersection of the invariant surface $\Sigma$ with the plane $x_k=x_i+x_j$
(the curves $l_1, l_2, l_3$ are depicted in the right panel of Figure~\ref{s_curves} in cyan color for all indices $\{i,j,k\}=\{1,2,3\}$).

Repeating similar procedures as in the case of the curves $s_k$ and $r_k$
we can obtain parametric equations~\eqref{kaler_curve}
of the curves $l_k$.
Clearly, $l_i\cap l_j=\emptyset$ for $i\ne j$.
We claim that each  of $l_1, l_2, l_3$  is also an invariant curve of the differential system~\eqref{three_equat}.
To show it consider  the case $k=3$ only due to symmetry.
Substitute the parametric representation
$x_1=\phi$, $x_2=t\phi$,  $x_3=t^{-1}\phi^{-2}$ of the curve $l_3$
into $f_1,f_2$ and $f_3$ in~\eqref{three_equat}, where
$\phi=\phi(t):=(t^2+t)^{-1/3}$, $t>0$.
Then
\begin{eqnarray*}
f_1=-\frac{2}{9}\, \frac{(2t+1)(t-1)}{t(t+1)}, \qquad
f_2=\frac{2}{9}\, \frac{(t+2)(t-1)}{t+1}, \qquad
f_3=\frac{2}{9}\, \frac{(t-1)^2}{t}.
\end{eqnarray*}
The value $t=1$ providing $f_1=f_2=f_3=0$ gives a   stationary trajectory,
namely it is the singular point $\bold{o}_3=\left(q, q, q \kappa\right)$ itself.
Thus assume $t\ne 1$.
The identities
\begin{eqnarray*}
\dfrac{dx_2}{dx_1}&=&\frac{\dfrac{dx_2}{dt}}{\dfrac{dx_1}{dt}}\equiv
\dfrac{(t\phi)'}{\phi'}=-\dfrac{(t+2)t}{2t+1}=\dfrac{f_2}{f_1}, \\ \dfrac{dx_3}{dx_1}&=&\dfrac{\big(t^{-1}\phi^{-2}\big)'}{\phi'}=
-\dfrac{t^2-1}{2t+1}=\dfrac{f_3}{f_1},\\
\dfrac{dx_3}{dx_2}&=&\dfrac{\big(t^{-1}\phi^{-2}\big)'}{(t\phi)'}=
\dfrac{t^2-1}{t(t+2)}=\dfrac{f_3}{f_2}
\end{eqnarray*}
imply that  $l_3$
is  a trajectory of~\eqref{three_equat}
for $t>0$ and $t\ne 1$.
Moreover, $l_3$ passes through the singular point~$\bold{o}_3$ as noted above.
This means that $l_3$ is a separatrice of~$\bold{o}_3$.
Invariancy of the curves $l_1$ and $l_2$ respectively
passing through $\bold{o}_1$ and $\bold{o}_2$
can be proved by the same way.
Theorem~\ref{theo24_2} is proved.
\end{proof}

\begin{remark}
 As  it was noted in the proof of Theorem~\ref{theo24_2}
 the equations~\eqref{param_low} define for $t\in (M,+\infty)$
 the same curve as~\eqref{param_up} for $t\in (0,m)$.
In the case $t\in (M,+\infty)$ the tail removing procedure leads
to the equation $t^{-1}\beta^{-2}=\beta$
equivalent to $at^2-(a^2+1)t+a=0$.
Its first root $t=a$ corresponds to the point $P_{ki}$
and the second root $t=1/a$ gives the point $P_{kj}$.
Obviously $0<a<m<M<1/a$ for all $a\in (0,1/2)$.
Therefore, in principle, both components of each curve $r_k$ can be parameterized
by one formula, say~\eqref{param_low}, but using the different intervals $(0,a]$ and  $[1/a,+\infty)$.
\end{remark}

\begin{remark}
We proved  that all singular points
$\bold{o}_0, \bold{o}_1, \bold{o}_2$ and $\bold{o}_3$
of the normalized Ricci flow on generalized Wallach spaces with
$a_1=a_2=a_3=a$
belong to the set
$\Sigma R$ of metrics with positive Ricci curvature.
Unfortunately a similar assertion does not hold for the
set~$\Sigma S$.  Lemma~3 in~\cite{Ab24} shows that there exists a critical value $a=3/14$ such that $\bold{o}_1, \bold{o}_2, \bold{o}_3\in  \Sigma S$
only if $a\in (3/14, 1/2)$ and  the
boundary cases $\bold{o}_i\in  s_i$ ($i=1,2,3$) hold if $a=3/14$.
The only generalized Wallach spaces which admit metrics with
positive sectional curvature are the Wallach spaces~\eqref{SWS} which satisfy the condition $a\in (0,3/14)$.
\end{remark}

\medskip

\begin{remark} The case $a=1/6$ is original,
where K\"{a}hler metrics provide separatrices of saddles~$\bold{o}_i$ as it illustrated
in Figure~\ref{s_curves}.
For $a \ne 1/6$ it is a difficult problem to find
similar separatrices analytically.
Knowing all separatrices allows to
predict the dynamics of the Ricci flow in more detail.
To demonstrate the main idea consider
an arbitrarily chosen  singular point in the case $a=1/6$, say
$\bold{o}_3=\big(2^{-1/3}, 2^{-1/3}, 2^{2/3} \big)$ (the K\"{a}hler-Einstein metric) and observe that the curve~$l_3$ defined by the equations $x_3=x_1+x_2$ and $x_1x_2x_3=1$ coincides with  the unstable manifold~$W_3^u$
of~$\bold{o}_3$ as it was shown in~\cite{Ab24}.
The  stable manifold of $\bold{o}_3$
is $W_3^s:=\big\{(x_1,x_2,x_3)\in \mathbb{R}^3~ \big|~ x_3=p^{-2},\, x_1=x_2=p, ~ 0<p<1\big\}\subset I_3$.
It is clear now that controlling by~$W_i^s$ and $W_i^u$  trajectories of~\eqref{three_equat} never can leave the domain bounded by the curves~\eqref{kaler_sep}.
This explains the fact proved in Theorem~4 in~\cite{AN}
that Riemannian metrics~\eqref{metric}
on generalized Wallach spaces with $a=1/6$ (on the Wallach space $\operatorname{SU}(3)/T_{\max}$ in particularly)
preserve the positivity of their Ricci curvature for $x_k<x_i+x_j$ ($\{i,j,k\}=\{1,2,3\}$).
In Figure~\ref{bime} the separatrices $l_1, l_2, l_3$ and some trajectories
of~\eqref{three_equat} are depicted  for illustrations.
\end{remark}

\begin{figure}[h]
\centering
\includegraphics[angle=0, width=0.90\textwidth]{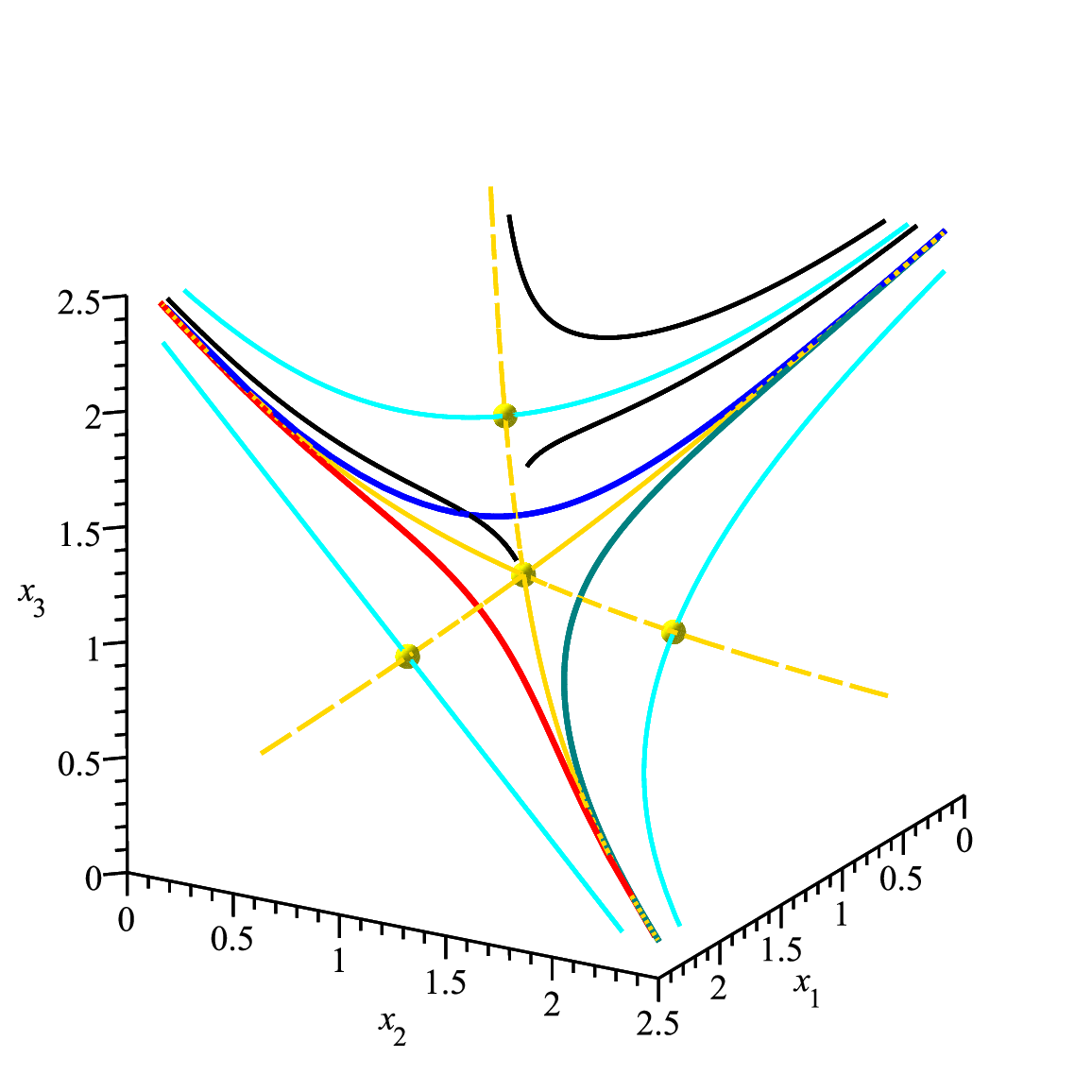}
\caption{The separatrices $l_1, l_2, l_3$ (in cyan color), $I_1, I_2, I_3$ (in yellow color) and  trajectories of~\eqref{three_equat} (in black color).}
\label{bime}
\end{figure}

\section{Additional remarks}

i) The well known fact  that the positivity of the Ricci curvature
follows from the positivity of the sectional curvature
can be  justified and illustrated via inclusion~$S\subset R$, where
$S$ is depicted in Figure~\ref{surfaces} as a set bounded by three cones in  red, teal and blue colors,
respectively~$R$ is bounded by six conic surfaces in magenta, aquamarine and burlywood colors.

To establish $S\subset R$ for all  $a\in (0,1/2)$
it suffices to show the inclusion~$\partial(S)\subset R$.
We will follow this opportunity since trying to establish $S\subset R$ directly
leads to  pairs of inequalities of the kind $\gamma_i>0$ and  $\lambda_i>0$
whose analysis  is much more complicated than to deal with the system consisting of
one equation $\gamma_i=0$
and one inequality  $\lambda_i>0$:
\begin{equation}\label{syst}
\aligned
\begin{cases}
(x_j-x_k)^2+2x_i(x_j+x_k)-3x_i^2=0, \\
x_jx_k+a\,\big(x_i^2-x_j^2-x_k^2\big)>0.
\end{cases}
\endaligned
\end{equation}

By symmetry fix any $i\in \{1,2,3\}$
and consider the component $\Gamma_i$ of the boundary of $S$.
Every point of the cone~$\Gamma_i$ belongs to some its generator line
$x_i=\nu t$, $x_j=\mu t$, $x_k=t$, $t>0$ (see also~\cite{Ab24}),
where
$\mu=1-\nu+2\sqrt{\nu(\nu-1)}>0$, $\nu>1$.
Indeed generators satisfy the equation in~\eqref{syst}
and the inequality in~\eqref{syst}
takes the form
$(X-Y)\,t^2>0$
with
$X:= \big(4a(\nu -1)+2\big)\sqrt{\nu(\nu-1)}$ and
$Y:=4a\nu^2+(1-6a)\nu +2a-1$.

Obviously $X>0$ for all $a\in (0,1/2)$ and $\nu>1$.
Since     $Y=0$ has roots $\nu_1=\frac{2a-1}{4a}<0$ and $\nu_2=1$
the inequality $Y>0$ holds as well at  $\nu>1$.
Thus  $X-Y>0$ is  equivalent to
$X^2-Y^2= (\nu-1)\, p(\nu)>0$, where
$p(\nu) =8a\nu^2-(2a+3)(2a-1)\nu+(2a-1)^2$
admits  two different negative roots
$\nu_1=\frac{2a-1}{16a} \left(2a+3+\sqrt{(2a-1)(2a-9)}\right)$
and $\nu_2=\frac{2a-1}{16a} \left(2a+3-\sqrt{(2a-1)(2a-9)}\right)$
for every $a\in (0,1/2)$.
It follows then  $p(\nu)>0$ (hence $X^2-Y^2>0$) at $\nu>1$ independently on $a\in (0,1/2)$.
This means that  $\lambda_i>0$ for any point  of $\Gamma_i$
which is equivalent to   $\Gamma_i \subset R$.
Since~$i$ was chosen arbitrarily we obtain
 $\partial(S)=\Gamma_1\cup \Gamma_2 \cup \Gamma_3 \subset R$ and hence $S\subset R$ with the obvious
consequence $\Sigma S \subset \Sigma R$.

\medskip

ii) There are useful asymptotical representations for practical aims.
For instance, at $t\to 0$ the expressions
$$
x_2(t)=t^{2/3}+\mathcal{O}\big(t^{8/3}\big), \quad
x_1(t)=x_3(t)=t^{-1/3}+\mathcal{O}\big(t^{5/3}\big)
$$
are valid for coordinates of points of the curve $s_3$ defined
as a variety of solutions of the system
\begin{equation}\label{plan_s}
\aligned
\begin{cases}
(x_1-x_2)^2+2x_3(x_1+x_2)-3x_3^2=0, \\
x_1x_2x_3=1.
\end{cases}
\endaligned
\end{equation}

For $t$ tending to $0$  the curve
$
r_1:
\aligned
\begin{cases}
x_2x_3+a\left(x_1^2-x_2^2-x_3^2\right)=0, \\
x_1x_2x_3=1
\end{cases}
\endaligned
$
has a similar asymptotic
 $$
x_2(t)=t^{2/3}+\frac{t^{5/3}}{6a}+\mathcal{O}\big(t^{8/3}\big), \quad
x_1(t)=t^{-1/3}+\mathcal{O}\big(t^{5/3}\big),
\quad
x_3(t)=t^{-1/3}+\frac{t^{2/3}}{6a}+\mathcal{O}\big(t^{5/3}\big)
$$
in accordance with the fact that  $s_3$ and $r_{12}\subset r_1$
approach the same invariant curve $I_2$ at infinity.

\medskip

iii) Often it is easier to deal with a planar analysis of the dynamics of the normalized Ricci flow.
Choose the coordinate plane $x_3=0$ without loss of generality. Then the projection of the set~ $\Sigma S$ of Riemannian metrics  with positive sectional curvature
onto the plane $x_3=0$
can be  bounded by the following plane curves $s_1', s_2'$ and $s_3'$ defined implicitly
\begin{eqnarray*}
3x_1^4x_2^2-2x_1^3x_2^3-x_1^2x_2^4-2x_1^2x_2+2x_1x_2^2-1&=&0,\\
3x_2^4x_1^2-2x_2^3x_1^3-x_2^2x_1^4-2x_2^2x_1+2x_2x_1^2-1&=&0,\\
x_1^4x_2^2-2x_1^3x_2^3+x_1^2x_2^4+2x_1^2x_2+2x_1x_2^2-3&=&0.
\end{eqnarray*}
For example the equation of $s_3'$ can be obtained eliminating $x_3$ in the system~\eqref{plan_s}.

Analogously, boundary curves of the projection of the set $\Sigma R$ of Ricci positive metrics onto the plane $x_3=0$ have equations
\begin{eqnarray*}
ax_1^4x_2^2-ax_1^2x_2^4+x_1x_2^2-a&=&0,\\
ax_2^4x_1^2-ax_2^2x_1^4+x_2x_1^2-a&=&0,\\
ax_1^4x_2^2+ax_1^2x_2^4-x_1^3x_2^3-a&=&0.
\end{eqnarray*}

Projections of the K\"{a}hler metrics
$x_1=x_2+x_3$, $x_2=x_1+x_3$ and $x_3=x_1+x_2$ will be defined by
$x_1x_2(x_1-x_2)=1$,\, $x_1x_2(x_2-x_1)=1$ and $x_1x_2(x_1+x_2)=1$ respectively.

We recommend to compare  the pictures demonstrated in this paper with planar pictures  depicted in the right panels of Figures~3, 6 and 7 obtained in \cite{AN} in the coordinate plane $(x_1,x_2)$.

\bigskip

The author is grateful to Prof. Yu.\,G.~Nikonorov for helpful discussions.

\vspace{10mm}

\bibliographystyle{amsunsrt}

\vspace{5mm}

\end{document}